\documentclass[12pt]{amsart}
\usepackage{amsmath,amssymb,amsbsy,amsfonts,amsthm,latexsym,mathabx,
            amsopn,amstext,amsxtra,euscript,amscd,stmaryrd,mathrsfs,
            cite,array,mathtools,enumerate}
            
\usepackage{url}
\usepackage[colorlinks,linkcolor=blue,anchorcolor=blue,citecolor=blue,backref=page]{hyperref}

\usepackage{color}

\usepackage[np]{numprint}
\npdecimalsign{\ensuremath{.}}

\renewcommand*{\backref}[1]{}
\renewcommand*{\backrefalt}[4]{%
    \ifcase #1 (Not cited.)%
    \or        (p.\,#2)%
    \else      (pp.\,#2)%
    \fi}
    
\DeclareMathOperator{\li}{li}

\DeclareMathOperator{\re}{Re}

\begin{document}

\newtheorem{theorem}{Theorem}
\newtheorem{Remark}{Remark}
\newtheorem{lemma}[theorem]{Lemma}
\newtheorem{claim}[theorem]{Claim}
\newtheorem{conj}[theorem]{Conjecture}
\newtheorem{cor}[theorem]{Corollary}
\newtheorem{prop}[theorem]{Proposition}
\newtheorem{definition}{Definition}
\newtheorem{question}[theorem]{Question}
\newcommand{\hh}{{{\mathrm h}}}

\numberwithin{equation}{section}
\numberwithin{theorem}{section}

\def\sssum{\mathop{\sum\!\sum\!\sum}}
\def\ssum{\mathop{\sum\ldots \sum}}

\def \balpha{\boldsymbol\alpha}
\def \bbeta{\boldsymbol\beta}
\def \bgamma{{\boldsymbol\gamma}}
\def \bomega{\boldsymbol\omega}

\newcommand{\Res}{\mathrm{Res}\,}
\newcommand{\Gal}{\mathrm{Gal}\,}

\def\sssum{\mathop{\sum\!\sum\!\sum}}
\def\ssum{\mathop{\sum\ldots \sum}}
\def\dsum{\mathop{\sum\  \sum}}
\def\iint{\mathop{\int\ldots \int}}

\def\squareforqed{\hbox{\rlap{$\sqcap$}$\sqcup$}}
\def\qed{\ifmmode\squareforqed\else{\unskip\nobreak\hfil
\penalty50\hskip1em\null\nobreak\hfil\squareforqed
\parfillskip=0pt\finalhyphendemerits=0\endgraf}\fi}%%

%  use the AMS-Euler Fraktur fonts
%%%%%%%%%%%%%%%%%%%%%%%%%%%%%%%%%%
\newfont{\teneufm}{eufm10}
\newfont{\seveneufm}{eufm7}
\newfont{\fiveeufm}{eufm5}
%%%%%%%%%%%%%%%%%%%%%%%%%%%%%%%%%
%
%  allow automatic size selection in math mode
%
%%%%%%%%%%%%%%%%%%%%%%%%%%%%%%%%%
\newfam\eufmfam
     \textfont\eufmfam=\teneufm
\scriptfont\eufmfam=\seveneufm
     \scriptscriptfont\eufmfam=\fiveeufm
%%%%%%%%%%%%%%%%%%%%%%%%%%%%%%%%%
%
%  \frak works on a single symbol at a time...
%
\def\frak#1{{\fam\eufmfam\relax#1}}

\def\fK{\mathfrak K}
\def\fT{\mathfrak{T}}

\def\fA{{\mathfrak A}}
\def\fB{{\mathfrak B}}
\def\fC{{\mathfrak C}}
\def\fD{{\mathfrak D}}
\def\fM{{\mathfrak M}}
\def\fR{{\mathfrak R}}

\newcommand{\sX}{\ensuremath{\mathscr{X}}}

\def\vec#1{\mathbf{#1}}
\def\dist{\mathrm{dist}}
\def\vol#1{\mathrm{vol}\,{#1}}

\def\squareforqed{\hbox{\rlap{$\sqcap$}$\sqcup$}}
\def\qed{\ifmmode\squareforqed\else{\unskip\nobreak\hfil
\penalty50\hskip1em\null\nobreak\hfil\squareforqed
\parfillskip=0pt\finalhyphendemerits=0\endgraf}\fi}

\def\sA{\mathscr A}
\def\sB{\mathscr B}
\def\sC{\mathscr C}
\def\sD{\Delta}
\def\sE{\mathscr E}
\def\sF{\mathscr F}
\def\sG{\mathscr G}
\def\sH{\mathscr H}
\def\sI{\mathscr I}
\def\sJ{\mathscr J}
\def\sK{\mathscr K}
\def\sL{\mathscr L}
\def\sM{\mathscr M}
\def\sN{\mathscr N}
\def\sO{\mathscr O}
\def\sP{\mathscr P}
\def\sQ{\mathscr Q}
\def\sR{\mathscr R}
\def\sS{\mathscr S}
\def\sU{\mathscr U}
\def\sT{\mathscr T}
\def\sV{\mathscr V}
\def\sW{\mathscr W}
\def\sX{\mathscr X}
\def\sY{\mathscr Y}
\def\sZ{\mathscr Z}

%%%%%%%%%%%%%%%%%%%%%%%%%
% Alphabet calligraphie %
%%%%%%%%%%%%%%%%%%%%%%%%%
\def\cA{{\mathcal A}}
\def\cB{{\mathcal B}}
\def\cC{{\mathcal C}}
\def\cD{{\mathcal D}}
\def\cE{{\mathcal E}}
\def\cF{{\mathcal F}}
\def\cG{{\mathcal G}}
\def\cH{{\mathcal H}}
\def\cI{{\mathcal I}}
\def\cJ{{\mathcal J}}
\def\cK{{\mathcal K}}
\def\cL{{\mathcal L}}
\def\cM{{\mathcal M}}
\def\cN{{\mathcal N}}
\def\cO{{\mathcal O}}
\def\cP{{\mathcal P}}
\def\cQ{{\mathcal Q}}
\def\cR{{\mathcal R}}
\def\cS{{\mathcal S}}
\def\cT{{\mathcal T}}
\def\cU{{\mathcal U}}
\def\cV{{\mathcal V}}
\def\cW{{\mathcal W}}
\def\cX{{\mathcal X}}
\def\cY{{\mathcal Y}}
\def\cZ{{\mathcal Z}}
\newcommand{\rmod}[1]{\: \mbox{mod} \: #1}

\def\vr{\mathbf r}

\def\e{{\mathbf{\,e}}}
\def\ep{{\mathbf{\,e}}_p}
\def\em{{\mathbf{\,e}}_m}
\def\en{{\mathbf{\,e}}_n}

\def\Tr{{\mathrm{Tr}}}
\def\Nm{{\mathrm{Nm}}}
\def\supp{{\mathrm{supp}}}

\def\rE{{\mathrm{E}}}

 \def\SS{{\mathbf{S}}}

\def\lcm{{\mathrm{lcm}}}

\def\({\left(}
\def\){\right)}
\def\fl#1{\left\lfloor#1\right\rfloor}
\def\rf#1{\left\lceil#1\right\rceil}

\def\mand{\qquad \mbox{and} \qquad}

\newcommand{\commF}[1]{\marginpar{%
\begin{color}{red}
\vskip-\baselineskip %raise the marginpar a bit
\raggedright\footnotesize
\itshape\hrule \smallskip F: #1\par\smallskip\hrule\end{color}}}

\newcommand{\commI}[1]{\marginpar{%
\begin{color}{blue}
\vskip-\baselineskip %raise the marginpar a bit
\raggedright\footnotesize
\itshape\hrule \smallskip I: #1\par\smallskip\hrule\end{color}}}

\newcommand{\commP}[1]{\marginpar{%
\begin{color}{magenta}
\vskip-\baselineskip %raise the marginpar a bit
\raggedright\footnotesize
\itshape\hrule \smallskip P: #1\par\smallskip\hrule\end{color}}}

\newcommand{\commO}[1]{\marginpar{%
\begin{color}{cyan}
\vskip-\baselineskip %raise the marginpar a bit
\raggedright\footnotesize
\itshape\hrule \smallskip O: #1\par\smallskip\hrule\end{color}}}

%%%%%%%%%%%%%%%%%%%%%%%%%%%%%%%%%%%%%%%%%%%%%%%%%%%%%%%%
%%%%%%%%%%%%%%%%%%%%%%%%%%%%%%%%%%%%%%%%%%%%%%%%%%%%%%%%
%%%%%%%%%%%%%%%%%%%%%%%%%%%%%%%%%%%%%%%%%%%%%%%%%%%%%%%%
%%%%%%%%%%%%%%%%%%%%%%%%%%%%%%%%%%%%%%%%%%%%%%%%%%%%%%%%

%%%%%%%  END OF STANDARD STUFF %%%%%%%%%

%%%%%%%%%%%%%%%%%%%%%%%%%%%%%%%%%%%%%%%%%%%%%%%%%%%%%%%%
%%%%%%%%%%%%%%%%%%%%%%%%%%%%%%%%%%%%%%%%%%%%%%%%%%%%%%%%
%%%%%%%%%%%%%%%%%%%%%%%%%%%%%%%%%%%%%%%%%%%%%%%%%%%%%%%%
%%%%%%%%%%%%%%%%%%%%%%%%%%%%%%%%%%%%%%%%%%%%%%%%%%%%%%%
%%%%%%%%%%%
%%% Spell

\hyphenation{re-pub-lished}
\hyphenation{ne-ce-ssa-ry}

\parskip 4pt plus 2pt minus 2pt

\def\bfdefault{b}
\overfullrule=5pt

\def \F{{\mathbb F}}
\def \K{{\mathbb K}}
\def \L{{\mathbb L}}
\def \Z{{\mathbb Z}}
\def \Q{{\mathbb Q}}
\def \R{{\mathbb R}}
\def \C{{\mathbb C}}
\def\Fp{\F_p}
\def \fp{\Fp^*}

\def \fP{\mathfrak P}
\def \fQ{\mathfrak Q}

\title{Denominators of Bernoulli polynomials}

\author[O. Bordell\`{e}s]{Olivier Bordell\`{e}s}
\address{O.B.: 2 All\'ee de la combe, 43000 Aiguilhe, France}
\email{borde43@wanadoo.fr}

\author[F. Luca] {Florian~Luca}

\address{F.L.:  School of Mathematics, University of the Witwatersrand, Private Bag X3, Wits 2050, South Africa; Max Planck Institute for Mathematics, Vivatsgasse 7, 53111 Bonn, Germany; 
Department of Mathematics, Faculty of Sciences, University of Ostrava, 30 Dubna 22, 701 03
Ostrava 1, Czech Republic}
\email{florian.luca@wits.ac.za}

\author[P.  Moree]{Pieter Moree}
\address{P.M.: Max Planck Institute for Mathematics, Vivatsgasse 7, 53111 Bonn, Germany}
\email{moree@mpim-bonn.mpg.de}

\author[I. E.  Shparlinski]{Igor E. Shparlinski}

\address{I.S.: Department of Pure Mathematics, University of New South Wales\\
2052 NSW, Australia.}
\email{igor.shparlinski@unsw.edu.au}

\begin{abstract}
For a positive integer $n$ let 
$$
\fP_n=\prod_{\substack{p\\ s_p(n)\ge p}} p,
$$
where $p$ runs over primes and  $s_p(n)$ is the sum of the base $p$ digits of $n$. For all $n$ 
we prove that $\fP_n$ is divisible by all ``small" primes with at most one exception. We also show that $\fP_n$ is large, 
has  many prime factors exceeding $\sqrt{n}$, 
with the largest one
exceeding $n^{20/37}$. We establish Kellner's conjecture, which says
that the number of prime factors exceeding $\sqrt{n}$ grows asymptotically as $\kappa \sqrt{n}/\log n$ 
for some constant $\kappa$ with $\kappa=2$.
Further, we  compare the sizes of $\fP_n$ and $\fP_{n+1}$, leading to the somewhat 
surprising conclusion that although 
$\fP_n$ tends to infinity with $n$, the inequality $\fP_n>\fP_{n+1}$
is more frequent than its 
reverse.
\end{abstract}

\maketitle

\section{Introduction}

\subsection{Motivation}
\label{sec:mot}
For positive integers $n$ and $b\ge 2$ let $s_b(n)$ be the sum of the base $b$-digits of $n$. The product
$$
\fP_n=\prod_{p\text{~prime}:~s_p(n)\ge p} p
$$
has been  introduced  by  Kellner and Sondow~\cite{KeSo}.
Although a priori this could be an infinite product, it is actually a finite product
which terminates for $p >(n+1)/2$, see~\cite{Kel,KeSo}.
%(it is also shown in~\cite{KeSo}
%that  $\fP_n$ is correctly defined by a finite products which terminates for $p >(n+2)/2$, see also~\cite{Kel}).  

The relevance of this quantity is due to its link with denominators of   {\it Bernoulli polynomials\/}
$$
B_n(X) = \sum_{k=0}^n \binom{n}{k} B_k X^{n-k}, 
$$
where $B_k$ is the $k^{\text{th}}$ Bernoulli number. We also define the polynomials  
$$
\widetilde B_n(X) = B_n(X) - B_n
$$
which are of interest due to their connection to power sums, 
namely we have
$$
\sum_{j=1}^{N-1} j^{n-1} = \frac{\widetilde B_n(N)}{N},
$$
see~\cite{Kel,KeSo,KeSoAP}. 
It is shown in~\cite{Kel,KeSo} that the denominator of the polynomial 
 $\widetilde B_n(X) $ is $\fP_n$, thus
$$
\fP_n\widetilde B_n(X)  \in \Z[X].
$$

The celebrated  von Staudt--Clausen theorem, see~\cite[Theorem~118]{HardyWright}, 
fully describes the denominator 
of  $B_n$ for an even $n$ as the product of prime $p$ with $p-1$ 
dividing $n$:
$$
\fQ_n = \prod_{\substack{p:~p-1\mid n}} p.
$$
Recall that $B_1 = 1$ and $B_n= 0$ and thus $\fQ_n =1$ for odd $n > 1$. 
One thus sees  that the denominator of   $ B_n(X)$ is 
$\lcm\left[ \fP_n, \fQ_n\right]$. 

In this paper, we prove some results about small and large prime factors of $\fP_n$. Kellner~\cite{Kel} has also introduced and
studied the decomposition 
$$
\fP_n=\fP_n^{-}\cdot \fP_n^{+},
$$
where
$$
 \fP_n^{-}=\prod_{\substack{p<{\sqrt{n}} \\ s_p(n)\ge p}} p\mand \fP_n^{+}=\prod_{\substack{p>{\sqrt{n}}\\ s_p(n)\ge p}} p.
$$
Note that the definitions of $\fP_n^{\pm}$, with strict inequalities on $p$ in both, are correct 
since $s_p(p^2)=1$ if $n = p^2$ with $p$ a prime.
Hence, $p\nmid \fP_n$,  even if $p={\sqrt{n}}$ holds for a prime $p$. 
Motivated by the link with Bernoulli 
polynomials Kellner~\cite{Kel}
has initiated the study of the arithmetic structure of $\fP_n^{-}$ and  $\fP_n^{+}$.

Let, as usual,  $\omega(m)$ and $P(m)$ be the number of distinct prime factors of $m$ and the largest prime factor of $m$, respectively. 

%In particular, Kellner~\cite[Conjecture~1]{Kel},  makes the following conjecture.
\begin{conj}\rm{(Kellner~\cite[Conjecture~1]{Kel}.)}
\label{conj:Kell1} For $n > 192$ we have 
$P(\fP_n) > \sqrt{n}$.
\end{conj}
Conjecture~\ref{conj:Kell1} is equivalent with the
conjecture that for $n>192$ we have $ \fP_n^{+}> 1$ and
is established, 
up to the numerical value of 
the threshold $192$, in a much  stronger form in Theorem~\ref{thm:2}. 

%Kellner~\cite[Conjecture~2]{Kel} also  makes the following conjecture.
\begin{conj}
\rm{(Kellner~\cite[Conjecture~2]{Kel}.)}
\label{conj:Kell2} There is  some absolute constant $\kappa > 0$ such that 
$$
 \omega(\fP_n^+)=(\kappa+o(1)) \frac{\sqrt{n}}{\log n}
$$
as $n\to\infty$. 
\end{conj}
In this paper we will show that
this conjecture is true 
with $\kappa=2$. Moreover, we provide
a much sharper estimate 
(Theorem \ref{thm:3})
for $\omega(\fP_n^+)$ than conjectured
by Kellner.

Aside from our results on these two 
conjectures,  we improve some of the results on $\fP_n$ and $\fP_n^\pm$ that have been given by Kellner~\cite{Kel}
and  also obtain several new results.
\par After introducing some further required 
notation in Section \ref{sec:notation},
we will state our results in 
Section \ref{sec:results}.

\subsection{Notation}
\label{sec:notation}
For a real number $x$ we write $\lfloor x\rfloor$ and $\{x\}$ for its integer and fractional parts, respectively.
\par For a positive integer $k$ and a positive real number $x$ we write $\log_k x$ for 
the iteratively defined function given by $\log_1 x=\max\{1,\ln x\}$, where $\ln x$ is a natural logarithm 
of $x$ and $\log_k x=\max\{1,\log_{k-1} x\}$ 
for $k\ge 2$. We will also use the functions 
$\e(x) = \exp(2 \pi i x)$
and $\psi(x) = x - \fl x - \frac{1}{2}$.

\par We recall the definitions of $\omega(m)$ and $P(m)$ from Section~\ref{sec:mot} as the number of distinct prime factors of $m$ and the largest prime factor of $m$, respectively. Another standard notation we use is $\pi(x)$  for the counting function of primes $p\le x$. 

We also define $\delta$ to be the {\it Erd{\H o}s--Ford--Tenenbaum constant\/}
\begin{equation}
\label{eq:ErdConst}
\delta=1-(1+\ln\ln 2)/\ln 2 = 0.08607 \ldots .
\end{equation} 

Throughout the paper, the letters $p$ and $q$ always denote a prime number. 
%For any positive integer $m$ and real numbers $1 \le a < b \le m$
%$$\omega(m;a,b) = \sum_{\substack{p \mid m \\ a < p \le b}} 1.$$

Define
\begin{equation}
   \label{eq:delta}
   \delta_c(x) = e^{- c (\log x)^{3/5} (\log_2 x)^{-1/5}} \quad \( x > e, \ c > 0 \). 
\end{equation}
Note that $\delta_c(x) \log x \ll \delta_{c_0}(x)$ for any $0 < c_0 < c$. 

The exponential integral is given by
\begin{equation}
   \label{eq:intexp}
   \rE_1 (x) = \int_x^\infty \frac{e^{-t}}{t} \, \mathrm{d}t \quad \( x > 0 \), 
\end{equation}
see~\cite[Eq.~(5.1.1)]{Abr}. 

As usual $A= O(B)$,  $A \ll B$, $B \gg A$ are equivalent to $|A| \le c |B|$ for some 
 {\it absolute\/}  constant $c> 0$, whereas $A=o(B)$ means that $A/B\to 0$.

\subsection{Results}
\label{sec:results}

Our first result shows that all ``small" primes $p$, with at most one exception, divide $\fP_n$, where ``small" depends on $n$ in a way which is made precise in the following statement.
\begin{theorem}
\label{thm:1}
For any fixed  $\varepsilon>0$ there exists $n_{\varepsilon}$ such that  for all $n\ge n_{\varepsilon}$, 
all  primes $p\le (1/2-\varepsilon)\log_2 n/\log_3 n$, with at most
one exception, divide $\fP_n$.
\end{theorem}
\begin{Remark}
As $2\nmid \fP_{2^n}$, we see that the
exceptional prime sometimes exists.
\end{Remark}
Next we obtain reasonably tight upper and lower bounds on the 
number of prime divisors and the largest prime divisor of $\fP_n$.

\begin{theorem}
\label{thm:2}
We have
%$$
%\frac{{\sqrt{n}}}{\log n}\ll \omega(\fP_n^+)\le \omega(\fP_n)\ll \frac{\sqrt{n}}{\log n}, 
%$$
%and 
$$
P(\fP_n)\ge  P(\fP_n^+)\gg n^{20/37}.
$$
\end{theorem}

This result implies that there exists $n_0$ 
such that Conjecture~\ref{conj:Kell1} is true
with $192$ replaced by $n_0$.
%On the other hand,  the upper bounds on $\omega(\fP_n)$ and  
%$\omega(\fP_n^+)$ of Theorem~\ref{thm:2} improve the bound of~\cite[Theorem~4]{Kel}. 

It is useful to recall that  we always have $ P(\fP_n) <n/2+1$  (see~\cite{Kel,KeSo}), 
and it is easy to see that for any prime $p$ we have $ P(\fP_{2p-1}) = p$. 

We establish a stronger form of  Conjecture~\ref{conj:Kell2}.
\begin{theorem}
\label{thm:3}
There exists $c>0$ such that, for any positive integer $n$ sufficiently large, 
$$
\omega(\fP_n^{+}) =  n \rE_1 \( \log \sqrt{n} \)  + O \( \sqrt{n} \, \delta_c \( \sqrt{n} \) \), 
$$
where the exponential integral is defined in~\eqref{eq:intexp} and the function $\delta_c$ is given in~\eqref{eq:delta}.
\end{theorem}

%Note that since $P(\fP_n^{-})<{\sqrt{n}}$, it %follows that $\omega(\fP_n^{-})\le %\pi({\sqrt{n}})\ll {\sqrt{n}}/\log n$. Hence, %Theorem~\ref{thm:3}  implies the  bounds
%\begin{align*}
%\frac{\sqrt{n}}{\log n}\ll \omega(\fP_n^{+}) \le %\omega(\fP_n) & \le \omega(\fP_n^{+})+ O\( %\frac{\sqrt{n}}{\log n}\) \\
%&=  n \rE_1 \( \log \sqrt{n} \)   + O\( %\frac{\sqrt{n}}{\log n}\) \ll \frac{\sqrt{n}}{\log %n}.
%\end{align*}

Successive integration by parts yield that for any positive integer $N\ge 1$
$$\rE_1(x) = \frac{e^{-x}}{x} \sum_{m=0}^{N-1} \frac{(-1)^m m!}{x^m} + (-1)^N N! \int_x^\infty \frac{e^{-t}}{t^{N+1}} \, \textrm{d}t,$$
from which we immediately deduce the following estimate.

\begin{cor}
\label{co:3.4}
For all positive integers $n ,N$, with $n$ sufficiently large,
$$
   \omega \( \mathfrak{P}_n^+ \) = \sum_{j=1}^N  \frac{  (-1)^{j-1} 2^j (j-1)! \sqrt{n}}{(\log n)^j}
%   \frac{2 \sqrt{n}}{\log n} - \frac{4 \sqrt{n}}{(\log n)^2} + \frac{16 \sqrt{n}}{(\log n)^3} - \frac{96 \sqrt{n}}{(\log n)^4} + \dotsb \\
%   &\qquad  + (-1)^{N-1} \frac{2^N (N-1)! \sqrt{n}}{(\log n)^N} 
   + O \( \frac{2^{N+1} N! \sqrt{n}}{(\log n)^{N+1}} \),
$$
in particular Conjecture~\ref{conj:Kell2} 
holds true with 
$\kappa=2$.
\end{cor}

%We now  establish  Conjecture~\ref{conj:Kell2} on average. 
%
%\begin{theorem}
%\label{thm:3}
%For $x\ge 3$, we have 
%$$
%\sum_{n \le x} \omega(\fP_n^{+}) =  \frac{4x^{3/2}} {3\log x} +O\(\frac{x^{3/2}\log_2 x}{(\log x)^2}\).
%$$
%\end{theorem}
%\begin{cor}
%If Conjecture~\ref{conj:Kell2} holds true,
%then $\kappa=2$.
%\end{cor}
%The corollary follows on
%noticing that Conjecture~\ref{conj:Kell2}  implies that, as $x\rightarrow \infty$,
%$$
%\sum_{n \le x} \omega(\fP_n^{+}) = \(\frac{2 \kappa}{3} +  o(1)\) \frac{x^{3/2}} {\log x}.  
%$$

The following estimate is 
derived in a similar way.

\begin{theorem}
\label{thm:3.5}
There exists a constant $c>0$ such that asymptotically 
$$
\log \fP_n^+ = \sqrt{n} + O \( \sqrt{n} \, \delta_c \( \sqrt{n} \) \),
$$
where the function $\delta_c$ is 
defined in~\eqref{eq:delta}.
\end{theorem}

%\begin{theorem}
%\label{thm:3.5}
%For $x\ge 3$, we have
%$$
%\sum_{n\le x} \log \fP_n^+=\frac{2}{3} x^{3/2}+O\(\frac{x^{3/2} \log_2 x}{\log x}\).
%$$
%\end{theorem}

%We derive  Theorem~\ref{thm:3.5} from a similar statement where the averaging is
%carried out over a short interval (see Lemma~\ref{lem:ShortInt2} below). This could be of independent 
%interest, especially as the ultimate goal is to eliminate any averaging. 
%Thus, reducing the length $y$ of the interval $(x,x+y]$ in 
%Lemma~\ref{lem:ShortInt2} is an interesting direction to pursue.

Finally, we look at how $\fP_n$ changes as we move from $n$ to $n+1$. Since $\fP_n$ tends to infinity with $n$
by Theorems~\ref{thm:1} and~\ref{thm:2}, it follows that the inequality $\fP_{n+1}>\fP_n$ holds infinitely often.
Surprisingly though, the reverse inequality is much more frequent and 
in fact even in a strict sense, namely 
we have $\fP_n>\fP_{n+1}$  with frequency about $\ln 2 = 0.6931 \ldots.$
However, we also show that the equality $\fP_n =  \fP_{n+1}$  holds for infinitely many $n$
as well.

\begin{theorem}
\label{thm:4} For any $x\ge 3$ we have:
\begin{itemize}
\item[(i)] the divisibility $\fP_{n+1}\mid \fP_n$ holds  for all except maybe at most 
$O(x (\log_2 x)^{-\delta} (\log_3 x)^{-1/2})$ positive integers $n\le x$;
\item[(ii)] the divisibility $\fP_{n+1}\mid \fP_n$ and the  inequality  $\fP_n > \fP_{n+1}$ hold simultaneously 
 for at least  $\(\ln  2 +o(1)\) x$  positive integers $n\le x$ as $x\to\infty$;  
\item[(iii)] the equality $\fP_q =  \fP_{q+1}$  holds for all except maybe at most 
$O(\pi(x)(\log_2 x)^{-c})$ primes $q\le x$, where $c > 0$ is an absolute constant. 
\end{itemize}
\end{theorem}

We remark that Kellner and Sondow~\cite[Theorem~4]{KeSoAP} have 
shown that for odd $n \geq 1$ the quotient $\fP_{n} / 
\fP_{n+1}$ is an odd integer, except that $\fP_{n} / 
\fP_{n+1} = 2$ if $n = 2^k - 1$ for some $k \geq 2$. One can also find 
in~\cite{KeSoAP} several more  results about the possible values of the 
ratios  $\fP_{n} / \fP_{n+1}$ for $n$ of special structure.

\subsection{Underlying techniques}

It is probably interesting to note that 
in our approach we use a combination of various elementary,  Diophantine and analytic  techniques.

In particular,  for the proof of Theorem~\ref{thm:1} we employ lower bounds 
of linear forms in logarithms due to Matveev~\cite{Mat}.

For the proof of Theorem~\ref{thm:2} we use a  result about the 
distribution of fractional parts of reciprocals of primes as well as bounds of exponential
sums with reciprocals due to Baker and Harman~\cite{BaHa1,BaHa2}.
%These results are combined with the Brun sieve (see~\cite{IwKow,Ten}) and the {\it Erd{\H o}s--Tur{\'a}n inequality\/} from  the theory 
%of uniform  distribution (see~\cite{DrTi}). 

Finally, we use a recent improvement due to 
 McNew,  Pollack and  Pomerance~\cite{NPP} of a result of Erd\H os and Wagstaff~\cite{EW} 
on the count of positive integers 
$n$ divisible by shifted primes (see also~\cite{Ford}), as well as a result 
of  Luca, Pizarro-Madariaga and  Pomerance~\cite{LPP} about shifted primes divisible
by another shifted prime.

\section{Proof of Theorem~\ref{thm:1}}

\subsection{Sums of digits  of integers in different bases}

Let $a,b\ge 2$ be fixed multiplicatively independent integers.
 It is shown by Senge and  Straus~\cite{SeSt}, that if $K$ is any fixed number, then there are only finitely many positive integers $n$ such that the sum of digits 
of $n$ in both bases $a$ and $b$ is at most $K$. This has been made effective by Stewart~\cite{Ste}
who, in particular, gives a lower bound 
\begin{equation}
\label{eq:Stewart}
s_a(n)+s_b(n)>\frac{\log_2 n}{\log_3 n+C(a,b)}-1
\end{equation}
for all $n>25$, where $C(a,b)$ is some constant depending on $a$ and $b$. 

The constant $C(a,b)$ is not made explicit in~\cite{Ste}. Here we  do so, as this is important
for our purposes, and may also be of independent interest.
As in~\cite{Ste}, our approach is based 
on lower bounds for linear forms in logarithms, where we use 
the bound of  Matveev~\cite{Mat}. We only  need it for logarithms of  rational numbers
rather than in its full generality for logarithms of algebraic numbers. 
We note that for us only the asymptotic dependence of $C(a,b)$ on $a$ and $b$ 
is important, but we also use this as an opportunity to derive a completely 
explicit expression for $C(a,b)$. 

Let $\rho=r/s$ be a 
rational number in reduced form (so, $(\gcd(r,s)=1$ and $s\ge 1$). 
Then its height is defined as 
$$
h(r/s)=\max\{\log |r|, \log s\}.
$$ 
Let $\alpha_1,\ldots,\alpha_k$ be rational numbers not zero or $\pm 1$. We put $A_i=h(\alpha_i)$ for $i=1,\ldots,k$. We let  $d_1,\ldots,d_k$ be nonzero integers and denote
$\max\{|d_1|,\ldots,|d_k|\}$ by $D$.
Let 
\begin{equation}
\label{eq:Lambda}
\Lambda=\prod_{i=1}^k \alpha_i^{d_i}-1.
\end{equation}
The result below follows from~\cite[Corollary~2.3]{Mat}, and the details have been 
worked out as~\cite[Theorem~9.4]{BMS}.

\begin{lemma}
\label{lem:Mat}
If $\Lambda\ne 0$, then 
$$
\log |\Lambda|>-1.4\cdot 30^{k+3} k^{4.5} (1+\log D)\prod_{i=1}^kA_i.
$$
\end{lemma}

We are now ready to present a more explicit version of inequality~\eqref{eq:Stewart}.

\begin{lemma}
\label{lem:explicit}
Assume that $a$ and $b$ are coprime
 integers $\ge 2$. Let $B\ge \max\{a,b\}$. 
Then the inequality
$$
s_a(n)+s_b(n)>\frac{\log_2 n}{\log_3 n+C(a,b)}
$$
holds with $C(a,b)=\log(2\cdot 10^{12} (\log B)^2)$ for all $n>\exp(10^{15} (\log B)^4)$.
\end{lemma}

\begin{proof}
We follow~\cite[pp.~66--69]{Ste} with the appropriate modifications. Consider the following $a$-ary and $b$-ary 
expansions  of  $n>a+b$:
\begin{align*}
n & =  a_1 a^{m_1}+a_2 a^{m_2}+\cdots+a_r a^{m_r},\quad a_i\in \{1,\ldots,a-1\},\quad 1\le i\le r,\\
n & = b_1 b^{\ell_1}+b_2 b^{\ell_2}+\cdots+b_a b^{\ell_t},\quad b_i\in \{1,\ldots,b-1\},\quad 1\le i\le t,
\end{align*}
where 
$$
m_1>\cdots>m_r\ge 0\quad {\text{\rm and}}\quad \ell_1>\ell_2>\cdots>\ell_t\ge 0.
$$ 
We put
$$
\vartheta=c_0\log_2 n,
$$
with $c_0$ an explicit constant depending on $B$ to be found later. We now consider  the intervals 
$$
\varTheta_1=(0,\vartheta],\quad \varTheta_2=(\vartheta,\vartheta^2],\quad \ldots,\quad  \varTheta_k=(\vartheta^{k-1},\vartheta^k],
$$
where $k$ satisfies the inequalities
\begin{equation}
\label{eq:k}
\vartheta^k\le \frac{\log n}{4\log B}<\vartheta^{k+1}.
\end{equation}
We assume $k\ge 1$ in~\eqref{eq:k}. We now show that for an appropriate $c_0$ and sufficiently large $n$ each interval $\varTheta_s$ contains either a term of the form $m_1-m_i$ for $i=2,\ldots, r$ or a term of the form $\ell_1-\ell_j$ 
for $j=2,\ldots,t$. Let us suppose that it is not so. Then there is $s$ with $1\le s\le k$ such that $\varTheta_s$ does not contain any $m_1-m_i$ and any $\ell_1-\ell_j$. So, let $u,~v$ be given by
\begin{eqnarray}
\label{eq:s}
&& m_1-m_u\le \vartheta^{s-1},\quad m_1-m_{u+1}\ge \vartheta^s;\\
 && \ell_1-\ell_v\le \vartheta^{s-1},\quad \ell_1-\ell_{v+1}\ge \vartheta^s.
 \end{eqnarray}
 Note that, since $n>B^2$,
 $$
 m_1= \fl{ \frac{\log n}{\log a}} \ge \frac{\log n}{\log a}-1>\frac{\log n}{4\log B}>\vartheta^k\ge \vartheta^s
 $$
and a similar inequality holds for $\ell_1$. Write
$$
n   =   a^{m_u} \alpha_a+\zeta_a \mand 
n  =  b^{\ell_v} \beta_b+\zeta_b, 
$$
with 
\begin{align*} 
\alpha_a =  a_1 a^{m_1-m_u}+\cdots+\alpha_a\mand
\beta_b =  b_1 b^{\ell_1-\ell_v}+\cdots+\beta_b.
\end{align*}
Clearly, $\zeta_a\in [0,a^{m_{u+1}+1})$ and $\zeta_b\in [0,b^{\ell_{v+1}+1})$. We then have
\begin{equation}
\label{eq:Matapp}
|\alpha_a a^{m_u}-\beta_b b^{\ell_v}|\le \max\{\zeta_a,\zeta_b\}\le
\max\left\{a^{m_{u+1}+1},b^{\ell_{v+1}+1}\right\}.
\end{equation}
We show that the left hand side of~\eqref{eq:Matapp} is nonzero. In order to do so, recall that
$$
m_1-m_u\le \vartheta^{s-1}<\vartheta^k<\frac{\log n}{4\log B},
$$
so 
\begin{align*}
m_u & =  m_1-(m_1-m_u)= \fl{ \frac{\log n}{\log a}}-\frac{\log n}{4\log B}\\
& \ge   \frac{\log n}{\log a}-1-\frac{\log n}{4\log B}\ge \frac{3\log n}{4\log a}-1>\frac{\log n}{2\log a},
\end{align*}
provided $n>B^4$. In particular, $a^{m_u}>{\sqrt{n}}$. A similar argument shows that $b^{\ell_v}>{\sqrt{n}}$. Since $a$ and $b$ are coprime, it follows that if the left hand side   of~\eqref{eq:Matapp} is zero, then 
$a^{m_u}\alpha_a=b^{\ell_v} \beta_b$, so $a^{m_u}\mid \beta_b$ and $b^{\ell_v}\mid \alpha_a$. In particular, both $\alpha_a$ and $a^{m_u}$ exceed ${\sqrt{n}}$, so 
$$
n=a^{m_u} \alpha_a+\zeta_a\ge a^{m_u} \alpha_a>{\sqrt{n}} \cdot {\sqrt{n}}=n,
$$
a contradiction.  Assuming that the maximum on the right hand side of~\eqref{eq:Matapp} is $b^{\ell_{v+1}+1}$, 
we divide both sides of~\eqref{eq:Matapp} by $B b^{\ell_v}$, getting
\begin{equation}
\label{eq:1}
\left|(A/B) a^{m_u} b^{-\ell_v}-1\right|<\frac{b^{\ell_{v+1}+1}}{b^{\ell_v} \beta_b}\le \frac{1}{b^{\ell_v-\ell_{v+1}-1}}.
\end{equation}
A similar inequality holds when the maximum on the right hand side of~\eqref{eq:Matapp} is $a^{m_{u+1}+1}$, namely
\begin{equation}
\label{eq:2}
\left|(B/A) a^{-m_u} b^{\ell_v}-1\right|<\frac{1}{a^{m_1-m_{u+1}-1}}.
\end{equation}
So, we are all set to apply Lemma~\ref{lem:Mat} to find a lower bound on the left hand side of~\eqref{eq:1} or~\eqref{eq:2}. 

We take 
$$
l=3,\quad \alpha_1=(A/B),\quad \alpha_2=am, \quad \alpha_3=b, 
$$ 
and 
$$
(d_1,d_2,d_3)=\varepsilon(1,m_u,-l_v),
$$
where $\varepsilon\in \{\pm 1\}$. More precisely, we set
\begin{itemize}
\item   $\varepsilon=1$ if in~\eqref{eq:Matapp} we have
$ \max\left\{a^{m_{u+1}+1},b^{\ell_{v+1}+1}\right\} = b^{\ell_{v+1}+1}$\\ (and thus~\eqref{eq:1} holds);

\item   $\varepsilon=-1$ if in~\eqref{eq:Matapp} we have
$ \max\left\{a^{m_{u+1}+1},b^{\ell_{v+1}+1}\right\} = a^{m_{u+1}+1}$ (and thus~\eqref{eq:2} holds).
\end{itemize}
Since by
assumption $a$ and $b$ are multiplicatively independent, 
these events are 
clearly mutually exclusive, and so
the choice of $\varepsilon$ is well-defined.

Obviously, we have 
$$
\max\{ m_u,\ell_v\}=\max\left\{ \fl{\frac{\log n}{\log a}}, \fl{\frac{\log n}{\log b}} \right\}\le \frac{\log n}{\log 2}<2\log n,
$$
so we can take $D=2\log n$. Furthermore, $A_2=\log a$ and $A_3=\log b$. As for $A_1$, we have
\begin{align*}
h(\alpha_1) & \le   \max\{\log \alpha_a, \log \beta_b\}\le \max\left\{\log(a^{m_1-m_u+1}), \log (b^{\ell_1-\ell_v+1})\right\}\\
& \le \max\left\{m_1-m_u+1,\ell_1-\ell_v+1\right\}\log B\le (\vartheta^{s-1}+1)\log B\\
& \le  2\vartheta^{s-1} \log B,
\end{align*}
so we take $A_1=2\vartheta^{s-1}\log B$. Now Lemma~\ref{lem:Mat} combined with inequalities~\eqref{eq:1} or~\eqref{eq:2} tells us that
\begin{align*}
 -1.4\times 30^6 \times 3^{4.5} (1 & +\log(2\log n)) 2\vartheta^{s-1}(\log a) (\log b )(\log B)\\
& <  -\max\{\log(a^{m_1-m_{u+1}-1}), \log (b^{\ell_1-\ell_{v+1}-1})\}.
\end{align*}
Assuming that $n>230$ (so that $2\log_2 n>1+\log(2\log n)$), we derive 
\begin{align*}
 \max\{(m_1&-m_{u+1}-1)\log a, (\ell_1-\ell_v-1)\log b\}\\
& <  6\cdot 10^{11}\vartheta^{s-1} (\log a)(\log b) (\log B)(\log_2 n).
\end{align*}
Since by inequality~\eqref{eq:s}
$$
\max\left\{m_1-m_{u+1}-1,\ell_1-\ell_{v+1}-1\right\}\ge \vartheta^{s}-1> \frac{\vartheta^s}{2}
$$
for $s\ge 1$, we see that
$$
\frac{\vartheta^s}{2}<6\cdot 10^{11} \vartheta^{s-1} (\log B)^2 \log_2 n,
$$
giving
$$
\vartheta<12\cdot 10^{11} (\log B)^2 \log_2 n,
$$
which is false if we choose $c_0=2\cdot 10^{12} (\log B)^2$. It remains to establish the starting value for $n$ such that $k\ge 1$. That is, 
$$
\vartheta<\frac{\log n}{4\log B}.
$$
This is equivalent to
\begin{equation}
\label{eq:c1}
\frac{\log n}{\log_2 n}>4c_0(\log B)^3.
\end{equation}
The right hand side above is $8\cdot 10^{12} (\log B)^3$. The inequality 
$$
\frac{x}{\log x}>A
$$ 
is fulfilled when $x>2A\log A$, provided 
that $A>e$. We now take $A=10^{13} (\log B)^3$ and we see  that~\eqref{eq:c1}
holds, provided that
$$
\log n> 2\cdot 10^{13} (\log B)^3(13\log 10+3\log_2 B).
$$
Since $\log_2 B<\log B$, $B\ge 3$, and $13\log 10+3<33$, it follows that the desired inequality holds for $\log n>10^{15} (\log B)^4$. 
Thus indeed, each interval $\varTheta_s$ for $s=1,\ldots,k$ contains one of $m_1-m_i$ or $\ell_1-\ell_j$ for $i=1,\ldots,r$, or $j=1,\ldots,t$. Hence, $r-1+s-1\ge k$. Therefore,
\begin{align*}
s_a(n)+s_b(n) & \ge  r+s\ge k+2=(k+1)+1\\
&>\frac{\log((\log n)/(4\log B))}{\log \vartheta}+1\\
& =  \frac{\log_2 n+(\log \vartheta-\log(4\log B))}{\log_3 n+\log c_0}, 
\end{align*}
on recalling that $\vartheta=c_0\log \log n$.
Clearly, $\vartheta>4\log B$. Thus the inequality
$$
s_a(n)+s_b(n)>\frac{\log_2 n}{\log_3 n+C}
$$
holds with $C=\log(2\cdot 10^{12} (\log B)^2)$, which concludes the proof.
\end{proof} 

\subsection{The proof of Theorem~\ref{thm:1}} Let $\varepsilon>0$ be arbitrary. 
We consider the primes $p\le B$ with $B=\log_2 n$. Then, as $n\rightarrow \infty$, we have 
$$
C=\log(2\cdot 10^{12} (\log B)^2)=2\log_4 n+O(1).
$$
Further, we need to check that $n>\exp(10^{15} (\log B)^4)$. This is equivalent to 
$\log n>10^{15} (\log_3 n)^4$, which holds for all $n>\exp(10^{18})$. Thus, assuming $n$ is this large, for any two 
distinct primes $p,q\le B$, it follows 
by inequality (\ref{eq:Stewart}) that, 
as $n\rightarrow \infty$, the inequality
$$
s_p(n)+s_q(n)>(1+o(1))\frac{\log_ 2 n}{\log_3 n}
$$
holds uniformly in primes $p,q\le B$. If
in particular we take $P=(1/2-\varepsilon)\log_2 n/\log_3 n$, it follows that for every $n>n_{\varepsilon}$, there is at most one prime $q\le \log_2 n$ such that
$$
s_q(n)<P.
$$
For all other primes $p\le \log_2 n$, we then have $s_p(n)\ge P$. So, if in addition we also have $p<P$, then $s_p(n)\ge P>p$, so $p\mid \fP_n$. Thus, indeed, for all $\varepsilon>0$, there exists $n_{\varepsilon}$ such that if $n>n_{\varepsilon}$,
then all primes $p<(1/2-\varepsilon)\log_2 n/\log_3 n$ divide $\fP_n$ with at most one exception.

\section{The proof of Theorem~\ref{thm:2}}

\subsection{Fractional parts of reciprocals of primes} 
We will use the following result from~\cite[Proposition~2]{BaHa1}:

\begin{lemma}
\label{lem:x/p}
For all $v$ and $w$ that satisfy
$$
v^{37/20}\le w\le v^2,
$$
and are sufficiently large, we have
$$
\#\left \{p~ :~2v<p<3v, \  \left\{\frac{w}{p}\right\} \ge 1- \frac{w}{16 v^2}\right \} \gg \frac{w}{v\log w}.
$$
\end{lemma}  

\subsection{Concluding the proof} 
We now derive one of our main technical results. 

\begin{lemma}
\label{lem:prim div}
For all $n$ and $v$ that satisfy 
$$
v^{37/20}\le n\le v^2,
$$
and are sufficiently large,
we have
$$
\#\left \{p~:~2v<p<3v, \  p\mid  \fP_n^+\right \} \gg \frac{n}{v\log n}.
$$
\end{lemma} 

\begin{proof}  Let $p$ be a prime 
counted in Lemma~\ref{lem:x/p} taken with   $w=n$, that is,
$$
1>\left\{\frac{n}{p}\right\}\ge 1-\frac{n}{16 v^2}. 
$$

Clearly, $p^2>4v^2>n$. Thus, writing $n$ in base $p$, we have $n=ap+b$, where 
$$
a=\fl{n/p} \ge n/p-1, 
$$
and
$$
b=p\left\{\frac{n}{p}\right\}\ge p-\frac{np}{16 v^2}.
$$
Thus, 
\begin{align*}
s_p(n) &=a+b\ge p + \frac{n}{p} -\frac{np}{16 v^2} -1
 = p + \frac{n}{p}\(1 -\frac{p^2}{16 v^2}\) -1 \\
 & > p + \frac{n}{p} \(1 -\frac{9}{16}\) -1=  
 p +\frac{7n}{16p} -1 >p ,
\end{align*}
assuming $p$ (and hence, $n$) is sufficiently large. 
\end{proof}
Now we are home and dry.
\begin{proof}[The proof of Theorem~\ref{thm:2}]
Apply  Lemma~\ref{lem:prim div} with 
$v = n^{20/37}$. 
\end{proof}

 \section{Proof of Theorems~\ref{thm:3} and~\ref{thm:3.5}}
 
 \subsection{Preliminary comments}

The proofs of both Theorems~\ref{thm:3} and~\ref{thm:3.5} may be unified in only one proof by 
making use of the number $\kappa \in \{0,1\}$. 
Indeed, our method also works if the number 
$\kappa$ is any real non-negative number. 
However, this
generalization does not seem to bring added
value and is left to the interested reader.

\subsection{Tools}
 We start with  the following  simple result.

\begin{lemma}
\label{le2}
Let $\kappa \in \{0,1\}$. For any real numbers $\varepsilon \in (0,1)$, $x\ge 2$ and $1 < x^{\varepsilon} \le y \le z < x$
$$\sum_{z < p \le x}  \(  \fl{ \frac{x+y}{p}} -  \fl{ \frac{x}{p}} \)  (\log p)^\kappa < 2\varepsilon^{-1}(y+1) (\log x)^{\kappa}.$$
\end{lemma}

\begin{proof} 
As $\log p\le \log x$ for $p\le x$, it
suffices to show that the inequality 
holds if $\kappa = 0$. We write
$$\sum_{z < p \le x} \(  \fl{ \frac{x+y}{p}} -  \fl{ \frac{x}{p}} \)
 = \sum_{z < p \le x} \sum_{x/p < m \le (x+y)/p}1 .
$$
Collecting together products $k= mp \in (x, x+y]$ and changing the order of summation 
we obtain 
$$\sum_{z < p \le x} \(  \fl{ \frac{x+y}{p}} -  \fl{ \frac{x}{p}} \) =
\sum_{x < k \le x+y}  \sum_{\substack{p \mid k \\ z < p \le x}} 1
\le \sum_{x < k \le x+y} \frac{\log k}{\log z}.
$$
Since for  $k  \in (x, x+y]$ we have 
$$
\frac{\log k}{\log z} \le \frac{\log (x+y)}{\log z} \le  \frac{\log (2x)}{\log y} \le  \frac{2}{\varepsilon},
$$
the result now follows.  
\end{proof}

\begin{cor}
\label{cor3}
Let $\kappa \in \{0,1\}$ and $\alpha \in \( n^{-7/16},1 \)$. Then
$$\sum_{\sqrt{n}/\alpha < p \le n}  \( \fl{\frac{n}{p} + \frac{n-p}{p(p-1)}} -  \fl{\frac{n}{p}} \)  (\log p)^\kappa \ll \alpha \sqrt{n} (\log n)^{\kappa}.$$
\end{cor}

\begin{proof}
If $p > \sqrt{n}/\alpha$, then
$$\alpha \sqrt{n} - \frac{n-p}{p-1}> \alpha \sqrt{n} - \frac{n-\sqrt{n}/\alpha}{\sqrt{n}/\alpha-1} = \frac{\sqrt{n}(1+\alpha)(1-\alpha)}{\sqrt{n}-\alpha} > 0,$$
where we used that $(n-p)/(p-1)= (n-1)/(p-1) - 1$ is a decreasing function for $p\ge2$.
Consequently, 
\begin{align*}
\sum_{\sqrt{n}/\alpha < p \le n} & \( \fl{\frac{n}{p} + \frac{n-p}{p(p-1)}} -  \fl{\frac{n}{p}} \)  (\log p)^\kappa\\
& \le \sum_{\sqrt{n}/\alpha < p \le n}  \( \fl{\frac{n + \alpha \sqrt{n}}{p}}-  \fl{\frac{n}{p}} \)  (\log p)^\kappa,
\end{align*}
and the proof is achieved on invoking Lemma~\ref{le2} with $x=n$, $y= \alpha \sqrt{n}$, $z = \sqrt{n}/\alpha$ and $\varepsilon =1/16$.
\end{proof}

\begin{lemma}
\label{le35}
Let $\kappa \in \{0,1\}$ and let  $M$ be a positive integer. 
%For a  prime $p$, we define 
%$$
%g(p) = 0 \qquad \text{or} \qquad g(p) = \frac{n-p}{p(p-1)}.
%$$
For any real valued function $g: \Z\to \R$ and
%Then for
 any positive integer $H$ we have 
\begin{align*}
   &  \left | \sum_{ M < p \le 2M} (\log p)^\kappa \psi \( \frac{n}{p} + g(p) \)  \right | \\
    & \qquad \ll \frac{M}{H(\log M)^{1-\kappa}}  + \sum_{h \le H} \frac{1}{h} \left | \sum_{M < p \le 2M} (\log p)^\kappa \e \( \frac{nh}{p} \) \e \( hg(p) \) \right |
\end{align*}
\end{lemma}

\begin{proof}
For any $0 < |t| < 1$ we put 
$\Phi(t) = \pi t (1-|t|) \cot (\pi t) + |t|$.
Note that $0 < \Phi(t) < 1$ for $0 < |t| < 1$.\\
\indent We follow the proof of~\cite[Corollary~6.2]{borl}. It 
follows from the result of Vaaler~\cite{Vaal}, 
which we use in the form given by~\cite[Theorem~6.1]{borl}, 
that for any real number $x\ge 1$ and 
any positive integer $H$,
\begin{equation}
   \psi(x) = - \sum_{0 < |h|  \le H} \Phi \left( \frac{h}{H+1} \right) \frac{\e(hx)}{2 \pi i h} + \mathcal{R}_H(x), \label{e1}
\end{equation}
where the error term $\mathcal{R}_H(x)$ satisfies
\begin{equation}
   \left | \mathcal{R}_H(x) \right |  \le \frac{1}{2H+2} \sum_{|h|  \le H} \left( 1 - \frac{|h|}{H+1} \right) \e(hx). \label{e2}
\end{equation}
The right-hand side of~\eqref{e2} 
does not look like, but is in fact a 
non-zero real number since it can be shown, 
cf. \cite[Exercise 3, p. 350]{borl}, that
$$\sum_{|h|  \le H} \left( 1 - \frac{|h|}{H+1} \right) \e(hx) = \frac{1}{H+1} \left | \sum_{h=0}^H \e(hx) \right |^2.$$
Using~\eqref{e1} with $x =n/p + g(p)$, multiplying by $(\log p)^\kappa$ and summing over all 
the  primes $p$ in $\left( M,2M \right] $ yields the estimate
$$
\sum_{ M < p  \le 2M}  (\log p)^\kappa \psi \left ( \frac{n}{p} + g(p) \right )  = \Sigma_1 + \Sigma_2,
$$
where
\begin{align*}
\Sigma_1 &= - \sum_{0 < |h|  \le H} \Phi \left( \frac{h}{H+1} \right) \frac{1}{2 \pi i h } \sum_{ M < p  \le 2M} (\log p)^\kappa \e \( \frac{nh}{p} + hg(p)\), \\
\Sigma_2 &  = \sum_{ M < p  \le 2M} (\log p)^\kappa \mathcal{R}_H \( \frac{n}{p} + g(p) \).
\end{align*}
Now
\begin{align*}
   \left | \Sigma_1 \right | & \le  \sum_{0 < |h|  \le H} \frac{1}{2 \pi |h|} \left | \sum_{ M < p  \le 2M} (\log p)^\kappa e \left( \frac{nh}{p} + hg(p) \right) \right | \\
   &= \frac{1}{\pi} \sum_{h  \le H} \frac{1}{h} \left | \sum_{ M < p  \le 2M} (\log p)^\kappa \e \( \frac{nh}{p} \) 
   e \( hg(p) \) \right |,
\end{align*}
and by~\eqref{e2} we have
\begin{align*}
   \left | \Sigma_2 \right | &  \le   \sum_{ M < p  \le 2M} (\log p)^\kappa \left | \mathcal{R}_H \left( \frac{n}{p} + g(p) \right) \right | \\
   &  \le   \frac{1}{2H+2} \sum_{|h|  \le H} \left( 1 - \frac{|h|}{H+1} \right) \sum_{ M < p  \le 2M} (\log p)^\kappa e \left( \frac{nh}{p} + hg(p) \right) \\
   &=  \frac{1}{2H+2} \sum_{M < p  \le 2M} (\log p)^\kappa + \frac{1}{H+1} \sum_{h  \le H} \left( 1 - \frac{h}{H+1} \right) \\
   &  \qquad \qquad \times \re \left( \sum_{ M < p  \le 2M} (\log p)^\kappa \e \( \frac{nh}{p} + hg(p) \) \right) \\
   &  \le  \frac{1}{2H+2} \sum_{M < p  \le 2M} (\log p)^\kappa \\
    &  \qquad \qquad + \sum_{h  \le H} \frac{1}{h} \left | \sum_{ M < p \le 2M} (\log p)^\kappa 
    \e \( \frac{nh}{p} \) \e \( hg(p) \) \right |, 
\end{align*}
where in the third line the cases $h=0$ and $|h| > 0$ 
are separated, concluding the proof.
\end{proof}

Recall that the von Mangoldt function is defined by
$$\Lambda(m)=
\begin{cases}
\log p&\qquad\text{if $m$ is a power of a prime $p$;} \\
0&\qquad\text{otherwise.}
\end{cases}
$$
 
\begin{lemma}
\label{le4}
Let $\kappa \in \{0,1\}$ and let $n\ge 1$ and  $M\ge 2$ be positive integers. 
For a  prime $p$, define
$$
g(p) = 0 \qquad \text{or} \qquad g(p) = \frac{n-p}{p(p-1)}.
$$
Then for any positive integer $H$ we have
\begin{align*}
   &  \left | \sum_{ M < p \le 2M} (\log p)^\kappa \psi \( \frac{n}{p} + g(p) \)  \right | \\
   & \quad  \ll  \frac{1}{(\log M)^{1-\kappa}}\sum_{h \le H} \frac{1}{h} \( 1 + \frac{nh}{M^2} \) \max_{M < N \le 2M} \left | \sum_{ M < m \le N} \Lambda(m) \e \( \frac{hn}{m} \) \right | \\
 & \qquad \quad  + \frac{M}{H(\log M)^{1-\kappa}} + \frac{nH}{M^{3/2}(\log M)^{1-\kappa}} + \frac{\sqrt{M} \log H}{(\log M)^{1-\kappa}}.
\end{align*}
\end{lemma}

\begin{proof}
From Lemma~\ref{le35},  
\begin{align*}
   &  \left | \sum_{ M < p \le 2M} (\log p)^\kappa \psi \( \frac{n}{p} + g(p) \)  \right | \\
    & \qquad \ll \frac{1}{H} \sum_{M < p \le 2M} (\log p)^\kappa + \sum_{h \le H} \frac{1}{h} \left | \sum_{M < p \le 2M} (\log p)^\kappa \e \( \frac{nh}{p} \) \e \( hg(p) \) \right | \\
    & \qquad \ll \frac{M}{H(\log M)^{1-\kappa}}  + \sum_{h \le H} \frac{1}{h} \left | \sum_{M < p \le 2M} (\log p)^\kappa \e \( \frac{nh}{p} \) \e \( hg(p) \) \right |.
\end{align*}
Now since 
$$\left | hg(p) \right | \le \frac{nh}{p(p-1)} \le \frac{2nh}{M^2} \quad \( M < p \le 2M \),$$
we get by Abel summation
\begin{align*}
   &  \left | \sum_{ M < p \le 2M} (\log p)^\kappa \e \( \frac{nh}{p} \) \e(hg(p)) \right | \\
   & \qquad \qquad  \ll  \( 1 + \frac{nh}{M^2} \) \max_{M < N \le 2M} \left | \sum_{ M < p \le N} (\log p)^\kappa 
   \e \( \frac{nh}{p} \) \right |,
\end{align*}
and by Abel summation again
\begin{align*}
   &  \left | \sum_{M < p \le N} (\log p)^\kappa \e \( \frac{nh}{p} \) \right | \\
   & \qquad \qquad  \le  \frac{2}{(\log M)^{1-\kappa}} \underset{M < L \le N}{\max} \left | \sum_{M \le p \le L} \( \log p \) \e \( \frac{nh}{p} \) \right | \\
   &  \qquad \qquad \ll  \frac{1}{(\log M)^{1-\kappa}} \( \underset{M < L \le N}{\max} \left | \sum_{M \le m \le L} \Lambda(m) \e \( \frac{nh}{m} \) \right | + \sqrt{N} \).
\end{align*}
The asserted estimate follows on putting everything
together.
\end{proof}

We also recall that using~\cite[Theorem~9]{gra} with $k=2$ we obtain
\begin{lemma}
\label{le5}
If $M \le x^{3/5}/5$, then, for any $N \in \( M,2M \right]$
$$\left | \sum_{M < m \le N} \Lambda(m) \e \( \frac{x}{m} \) \right | < 17 \( x^2 M^{19} \)^{1/24} (\log 16M)^{11/4}.$$
\end{lemma}

\begin{lemma}
\label{le6}
Let $\kappa \in \{0,1\}$. There exists some absolute constant $c >0 $ such that, for any large real number $t > 1$ we have
$$\sum_{p > t} \frac{(\log p)^\kappa}{p(p-1)} = F_\kappa(t) + O \( t^{-1} \delta_c(t) \),$$
where $F_0(t) = \rE_1(\log t)$ and $F_1(t) = t^{-1}$.
\end{lemma}

\begin{proof}
We recall the Prime Number Theorem 
in the form
$$
\pi(u) = \li(u) + O\(u \delta_{c_0}(u)\),
$$
where $\li(u)$ is the  logarithmic integral 
$$
\li(u) = \int_{2}^u   \frac{\mathrm{d}t}{\log t} \quad \( u > 0 \), 
$$
the function $\delta_{c_0}$ is defined in 
\eqref{eq:delta}
and $c_0>0$ is an absolute constant (see~\cite[Theorem~12.2 and Eq.~(12.27)]{Ivi} or~\cite[Corollary~8.30]{IwKow}, for instance).
By partial summation and the Prime Number 
Theorem in the above form, we derive
\begin{align*}
   \sum_{p > t} & \frac{(\log p)^\kappa}{p(p-1)} = \sum_{p > t} \frac{(\log p)^\kappa}{p^2} + \sum_{p > t} \frac{(\log p)^\kappa}{p^2(p-1)} \\
   &\quad= - \frac{\pi(t)(\log t)^\kappa}{t^2} + \int_t^\infty \frac{\pi(u)\( 2 (\log u)^\kappa - \kappa (\log u)^{\kappa-1} \)}{u^3} \, \textrm{d}u \\
   &\quad \qquad + O \( \frac{1}{t^2 (\log t)^{1-\kappa}} \) \\
   &\quad= - \frac{(\log t)^\kappa}{t^2} \( \li(t) + O \( t \delta_{c_0}(t) \) \) \\
   &\quad  \qquad +  \int_t^\infty \frac{2 (\log u)^\kappa - \kappa (\log u)^{\kappa-1}}{u^3} \Big( \li(u) + O \( u \delta_{c_0}(u) \) \Big) \, \textrm{d}u \\
   &\quad= G(t) + O \( \frac{(\log t)^\kappa \delta_{c_0}(t)}{t} \),
\end{align*}
where
$$G(t)=- \frac{\li(t)(\log t)^\kappa}{t^2} + \int_t^\infty \frac{\li(u)\big( 2 (\log u)^\kappa - \kappa (\log u)^{\kappa-1} \big)}{u^3} \, \textrm{d}u.$$
 Now, integrating by parts we derive
$$ G(t) =  \int_t^\infty \frac{\textrm{d}u}{u^2 (\log u)^{1- \kappa}} = F_\kappa(t).$$
The result follows with any constant $c \in \( 0,c_0 \)$.
\end{proof}

\subsection{First step}

Let $\kappa \in \{0,1\}$ and define
$$\Sigma_\kappa = \sum_{\sqrt{n} < p \le n}
 \( \fl{\frac{n-1}{p-1}} - \fl{\frac{n}{p}} \)  (\log p)^\kappa.$$

Notice that, if $\sqrt{n} < p \le n$, then
$$0 \le \frac{n-p}{p-1} < \sqrt{n},$$
and hence,
$$\frac{n-1}{p-1} - \frac{n}{p} = 
\frac{n-p}{p(p-1)} < \frac{\sqrt{n}}{p} < 1,$$
so that
\begin{equation}
\label{eq:omega Sigma}
\omega \( \mathfrak{p}_n^+ \) = \Sigma_0 \quad \textrm{and} \quad \log~(\mathfrak{p}_n^+) = \Sigma_1.
\end{equation}
We split the sum into two subsums as
 $$
 \Sigma_\kappa =  S_1 + S_2,
 $$
 where 
\begin{align*}
 S_1 &= \sum_{\sqrt{n} < p \le \sqrt{n} \, \( \delta_c \( \sqrt{n} \) \)^{-1}} \( \fl{\frac{n-1}{p-1}} - \fl{\frac{n}{p}} \)  (\log p)^\kappa, \\
 S_2& =  \sum_{\sqrt{n} \, \( \delta_c \( \sqrt{n} \) \)^{-1} < p \le n}   \( \fl{\frac{n-1}{p-1}} - \fl{\frac{n}{p}} \)  (\log p)^\kappa, 
 \end{align*}
and  $c >0$ is the constant given in Lemma~\ref{le6}.

\subsection{The sum $S_2$}

We use Corollary~\ref{cor3} with $\alpha = \delta_c \( \sqrt{n} \)$ obtaining immediately
\begin{equation}
   S_2 \ll \sqrt{n} \, \delta_c \( \sqrt{n} \) (\log n)^\kappa \ll \sqrt{n} \, \delta_{c_1} \( \sqrt{n} \) \label{e3}
\end{equation}
for some $c_1 \in \( 0,c \right]$.

\subsection{The sum $S_1$}

Using 
$$
\frac{n-1}{p-1} =  \frac{n}{p} +   \frac{n-p}{p(p-1)}, 
$$
and recalling the definition of the function $\psi(x)$, we write 
$$
   S_1 =S_{11} - S_{12}, 
$$
where 
\begin{align*}
S_{11} &= \sum_{\sqrt{n} < p \le \sqrt{n} \, \( \delta_c \( \sqrt{n} \) \)^{-1}} \frac{(\log p)^\kappa(n-p)}{p(p-1)}, \\
S_{12} & =
 \sum_{\sqrt{n} < p \le \sqrt{n} \, \( \delta_c \( \sqrt{n} \) \)^{-1}} (\log p)^\kappa \( \psi \( \frac{n}{p} 
 +  \frac{n-p}{p(p-1)}\) - \psi \( \frac{n}{p} \) \) .
\end{align*}

\subsubsection*{The main term}
For any integer $n \ge 1$, we derive from Lemma~\ref{le6} that
$$
   \sum_{\sqrt{n} < p \le \sqrt{n} \( \delta_c (\sqrt{n}) \)^{-1}} \frac{(\log p)^\kappa}{p(p-1)} = F_\kappa \( \sqrt{n} \) - F_\kappa \(\frac{\sqrt{n}}{\delta_c (\sqrt{n})}\) + O \( \frac{\delta_c (\sqrt{n})}{\sqrt{n}} \).
$$
From~\cite[Eq.~(5.1.19)]{Abr}, we have the inequalities
$$\frac{e^{-x}}{x+1} < \rE_1 (x) < \frac{e^{-x}}{x} \quad \( x > 0 \),$$
which imply that
$$F_0 \( \frac{ \sqrt{n}}{ \delta_c (\sqrt{n})}   \) \ll  \frac{\delta_c \( \sqrt{n} \)}{\sqrt{n}\log(n/\delta_c (\sqrt{n})^2)} \ll  \frac{\delta_c(\sqrt{n})}{\sqrt{n}},$$
and also
$$F_0 \( \frac{ \sqrt{n}}{ \delta_c (\sqrt{n})}   \) \ll \frac{\delta_c(\sqrt{n})}{\sqrt{n}},$$
so that
$$\sum_{\sqrt{n} < p \le \sqrt{n} \( \delta_c (\sqrt{n}) \)^{-1}} \frac{(\log p)^\kappa}{p(p-1)} = F_\kappa \( \sqrt{n} \) + O \( 
\frac{\delta_c(\sqrt{n})}{\sqrt{n}} \).$$
Therefore,
\begin{equation}
\begin{split}
\label{e4}
   S_{11} &= n \sum_{\sqrt{n} < p \le \sqrt{n} \( \delta_c (\sqrt{n}) \)^{-1}} \frac{(\log p)^\kappa}{p(p-1)} - \sum_{\sqrt{n} < p \le \sqrt{n} \( \delta_c (\sqrt{n}) \)^{-1}} \frac{(\log p)^\kappa}{p-1}  \\
   &= n \sum_{\sqrt{n} < p \le \sqrt{n} \( \delta_c (\sqrt{n}) \)^{-1}} \frac{(\log p)^\kappa}{p(p-1)} + O \( \log n \)  \\
   &= n F_\kappa \( \sqrt{n} \) + O \( \sqrt{n}\, \delta_c \( \sqrt{n} \) \). 
 \end{split} 
 \end{equation}

\subsubsection*{The error term}
It remains to prove that, for $n$ sufficiently large,
\begin{equation}
   \left | S_{12} \right | \ll \sqrt{n}\, \delta_c \( \sqrt{n} \). \label{e5}
\end{equation}
This estimate follows from the next result, which actually gives  a power saving.

\begin{lemma}
\label{le7}
We have
$S_{12}\le n^{49/100+o(1)}$ as $n\rightarrow \infty$. 
\end{lemma}

\begin{proof}
Split the interval $\( \sqrt{n} , \sqrt{n} \, \( \delta_c (\sqrt{n}) \)^{-1} \right]$ into $O (\log n)$ dyadic subintervals of the
form $\( M,2M\right]$, so that
\begin{align*}
   &  \left | \sum_{\sqrt{n} < p \le \sqrt{n} \, \( \delta_c (\sqrt{n}) \)^{-1}} (\log p)^\kappa \psi \( \frac{n}{p} + g(p) \) \right | \\
   &\qquad \qquad \ll   \underset{\sqrt{n} < M \le \sqrt{n} \, \( \delta_c (\sqrt{n}) \)^{-1}}{\max} \left | \sum_{M < p \le 2M} (\log p)^\kappa \psi \( \frac{n}{p} + g(p) \) \right | \log n,
\end{align*}
where, as in Lemma~\ref{le4}, 
$$
g(p) = 0 \qquad \text{or} \qquad g(p) = \frac{n-p}{p(p-1)}.
$$ 
From Lemmas~\ref{le4} and~\ref{le5}
\begin{align*}
   &  \left | \sum_{M < p \le 2M} (\log p)^\kappa \psi \( \frac{n}{p} + g(p) \) \right | \\
   &\qquad \  \ll   (\log M)^{7/4+\kappa}\sum_{h \le H} \frac{1}{h}\( 1 + \frac{nh}{M^2} \) \( h^2n^2 M^{19} \)^{1/24} \\
   &\qquad \qquad \qquad  + \frac{M}{H(\log M)^{1-\kappa}} + \frac{nH}{M^{3/2}(\log M)^{1-\kappa}} + \frac{\sqrt{M} \log H}{(\log M)^{1-\kappa}} \\
   &\qquad   \ll   \( \( n^{26}  H^{26} M^{-29} \)^{1/24} + \( n^2H^2M^{19} \)^{1/24} \) (\log M)^{11/4} \\
   & \qquad \qquad \qquad   + \frac{M}{H(\log M)^{1-\kappa}} +\frac{nH}{M^{3/2}} + \sqrt{M} \log H.
\end{align*}
Choose
$$H = \fl{M^{53/50} n^{-13/25}} $$
to balance the first and the third terms in the  above bound 
up to logarithmic factors, which we all replace by $n^{o(1)}$.
We also  note that, since $M \ge \sqrt{n}$, we have $H \ge 1$ provided 
that $n$ is large enough. Hence, 
\begin{align*}
%   &  \left | \sum_{M < p \le 2M} (\log p)^\kappa \psi \( \frac{n}{p} + g(p) \) \right | \\
%   & \qquad \qquad \ll  \( n^{26} M^{-3} \)^{1/50} (\log M)^{33/25} + \( n M^{22} \)^{1/25} (\log M)^{66/25} \\
%   &\qquad \qquad \qquad \qquad  + \( n^{12} M^{-11} \)^{1/25} (\log M)^{-33/25} + M^{1/2} \log M
   &  \left | \sum_{M < p \le 2M} (\log p)^\kappa \psi \( \frac{n}{p} + g(p) \) \right | \\
   & \qquad \quad \ll \( \( n^{26} M^{-3} \)^{1/50} + \( n M^{22} \)^{1/25}   + \( n^{12} M^{-11} \)^{1/25}  + M^{1/2} \)n^{o(1)}. 
\end{align*}
Using $M = n^{1/2 + o(1)}$ we obtain
$$
   \left | S_{12} \right |  \ll  \( n^{49/100}    + n^{13/50} \ + n^{1/4}\)n^{o(1)} 
 \ll  n^{49/100+ o(1)},
$$
as $n\rightarrow \infty$, 
concluding the proof.
\end{proof}

\subsection{Completion of the proof of Theorems~\ref{thm:3} and~\ref{thm:3.5}}

Follows at once from~\eqref{eq:omega Sigma}, \eqref{e3}, \eqref{e4} and~\eqref{e5}.

\section{The proof of Theorem~\ref{thm:4}}

\subsection{Preliminary considerations} 
\label{sec:prelim}

Let us fix some prime $p$ and  write $n$ in base $p$ as
$$
n=\sum_{i=0}^k a_i p^{k-i},
$$
with $p$-ary digits
$$
a_i\in \{0,\ldots,p -1\}, \qquad i=0, \ldots, k, \quad \text{and}\quad a_0\ne 0.
$$
In the above, $k=\fl{\log n/\log p}+1$.

We distinguish the following two cases:

\subsubsection*{Case 1:} $a_k\ne p-1$.

In this  case, 
$$
n+1=a_0p^{k}+\cdots+(a_k+1)
$$ 
is the base $p$-representation of $n+1$. Hence, $s_p(n+1)=s_p(n)+1$. 
In particular, the following events
\begin{equation}
\label{eq:div Pn}
p \mid  \fP_{n+1} \mand p \nmid \fP_n,
\end{equation}
and  thus, 
$$s_p(n+1) \ge  p >s_n(p),$$ 
are simultaneously   possible  only when 
\begin{equation}
\label{eq:snp p-1}
s_p(n) = p-1. 
\end{equation}

\subsubsection*{Case 2:} $a_k=p-1$.

Let $i\in [0,k]$ be such that $a_k=a_{k-1}=\cdots=a_{k-i}=p-1$, but $a_{k-i-1}<p-1$. Then 
$$
n+1=a_0p^k+\cdots+(a_{k-i-1}+1)p^{i},
$$ 
and we obtain
$$
s_p(n+1)  =  a_0+\cdots+a_{k-i-1}+1  \le  a_0+\cdots+a_{k-i-1}+a_k \le s_p(n).
$$
Hence, for each $p$ with $a_k = p-1$, if $p\mid  \fP_{n+1}$, then 
we also have $p\mid  \fP_n$. However, the opposite of~\eqref{eq:div Pn}, that is
\begin{equation}
\label{eq:div Pn+1}
p \nmid  \fP_{n+1} \mand p \mid \fP_n,
\end{equation}
 is also possible 
in this case. Actually this case plays an important role in our argument in Section~\ref{sec:proofi} below.

\subsection{Proof of~(i): divisibility}
\label{sec:proofi}

As we have seen, primes $p$ which belong to Case~2 of Section~\ref{sec:prelim}
do not appear in the denominator of the ratio $\fP_n/\fP_{n+1}$. We now show that 
for almost all $n\le x$ (with an explicit bound on the size of the exceptional set)
the primes which belong to Case~1 of Section~\ref{sec:prelim} do not appear 
in the denominator of this ratio either, and thus we have the desired divisibility.

More precisely, using the characterization of~\eqref{eq:div Pn} given by~\eqref{eq:snp p-1}, we 
 we conclude that  $\fP_{n+1}\mid \fP_n$ holds for all positive integers $n$ that have the property that there is no prime $p$ with
$s_p(n)=p-1$. It remains to prove that the complementary set has asymptotic density zero. 

Thus, we define
$$
\cA(x) =\{n\le x~:~s_p(n)=p-1~{\text{\rm for~some~prime}}~p\}.
$$

 First of all note that since $s_p(n)=p-1$, it follows that $(p-1)\mid n$. 
 McNew,  Pollack and  Pomerance~\cite{NPP}, improving on the previous result of 
 Erd\H os and Wagstaff~\cite{EW},  have shown that uniformly in $3\le y\le x$, 
\begin{equation}
\label{eq: p-1 | n}
\#\{n \le x~:~ (p-1) \mid n \ \text{ for some}\ p\ge y\} \ll \frac{x}{(\log y)^{\delta} (\log_2 y)^{1/2}},
\end{equation}
where $\delta$ is the Erd{\H o}s--Ford--Tenenbaum constant defined in ~\eqref{eq:ErdConst}.
We also recall that Ford~\cite{Ford} has recently established more precise results, which 
however cannot be used to improve our bounds. 

We take $y=\sqrt{\log x}$ in~\eqref{eq: p-1 | n}, getting that the number of $n\in \cA(x)$ such that $s_p(n)=p-1$ for some $p\ge y$ is 
\begin{equation}
\label{eq: large p}
 \frac{x}{(\log y)^{\delta} \sqrt{\log_2 y}} \ll \frac{x}{(\log_2 x)^{\delta} \sqrt{\log_3 x}}.
\end{equation}

Assume now that $p<y$. We remark that for each $p$ and
$$
k = \fl{\frac{\log n}{\log p}} +1,
$$ 
the condition $s_p(n)=p-1$ leads to 
the equation 
$$
a_0 +  \cdots +a_k=p-1
$$
on the $p$-ary digits $a_0, \ldots, a_k$ of $n$ (we possibly append some leading zeros to make all 
$p$-ary expansions of the same length). 
Thus, for each $p$ there are at most 
$$
\binom{k+p}{p} \le (k+p)^p \le (p +\log x)^p
$$
possible values for the string of digits $(a_0, \ldots, a_k)$, and therefore  for the number of such $n\in {\mathcal A}(x)$.
Hence, the total contribution from all $p \le y$ to $\#{\mathcal A}(x)$ is at most 
\begin{equation}
\label{eq: small p}
\sum_{p \le y}  (p+ \log x)^p \ll  (y + \log x)^y \le \exp\(2\,\sqrt{\log x}\, \log_2 x\),
\end{equation}
which is negligible when compared 
with~\eqref{eq: large p}, and concludes the proof.

\subsection{Proof of~(ii): divisibility and strict inequality}
\label{sec:proofii}
Since we have established that $\fP_{n+1}\mid \fP_n$ for almost all $n \le x$, 
it is enough to produce a sequence of $n$ of asymptotic density $\ln 2$ 
such that for each of them there exists a prime $p$ with the property~\eqref{eq:div Pn+1}.

Let $n\le x$ be of the form $n=ap-1$, where $p>\sqrt{x+1}$ and $a\ge 2$. Then $a\le (x+1)/p<\sqrt{x+1}< p$. 

In particular, $n+1=ap$ with $a<p$, so 
\begin{equation}
\label{eq: small s n+1}
s_p(n+1)=a<p.
\end{equation}
On the other hand, since $n=(a-1)p+p-1$, and $2\le a< p$, it follows that
\begin{equation}
\label{eq: large s n}
 s_p(n)=a-1+(p-1)=(a-2)+p\ge p.
\end{equation}
Comparing~\eqref{eq: small s n+1} and~\eqref{eq: large s n}, 
we see that we have~\eqref{eq:div Pn+1}.
Obviously, each integer $n \le x$ of this form can be generated by
only one prime  $p>\sqrt{x+1}$.  Fixing $p$, we have that 
$2\le a\le (x+1)/p$. Hence, there are $\fl{(x+1)/p}-1 = x/p + O(1)$ possibilities for $a$.
Thus, using the Mertens formula (see~\cite[Theorem~427]{HardyWright}),
 we see that we generate
 $$
\sum_{\sqrt{x+1}<p\le x+1} \( x/p + O(1)\) =  %x\sum_{\sqrt{x+1}<p<x+1} \frac{1}{p}+O(\pi(x))\\
%& =  x\ln 2+O(\pi(x))\\
%& =  x\ln 2+O \( \frac{x}{\log x} \) \\
 x \ln 2+o(x)
$$
such integers $n \le x$ as $x\to\infty$ in this way, concluding the proof.

\subsection{Proof of~(iii): equality}
We take $n=q-1$ for a prime $q \le x$. Clearly, 
$$s_q(n)=q-1<q \mand  s_q(n+1)=1<q.
$$ 
Thus,
$$
q\nmid \fP_n\fP_{n+1}.
$$
We remark that $q$ is the only prime that divides $n+1$. Thus,  Case~2
of Section~\ref{sec:prelim} is impossible for any prime $p\ne q$. 
Hence, the only primes   
in which $\fP_n$ and $\fP_{n+1}$ may differ are the  primes $p$ as in Case~1 
of Section~\ref{sec:prelim} for which $s_p(n)=p-1$. In particular, $(p-1)\mid (q-1)$ with $p\ne n$. It is shown in~\cite{LPP} that there exists a positive constant $c$ such that uniformly in $2\le y\le x$, 
we have an analogue of~\eqref{eq: p-1 | n} for shifted primes $q-1$ in place of $n$; that is, 
\begin{equation}
\label{eq: p-1 | q-1}
\#\{q \le x~:~ (p-1) \mid (q-1) \ \text{ for some}\ p\ge y\} \ll \frac{x}{\log x(\log y)^{c}},
\end{equation}
Again we take $y=\sqrt{\log x}$  and proceeding as in the last part of 
Section~\ref{sec:proofi},
in particular, using~\eqref{eq: small p} and thus ignoring the fact that $n = q-1$ is a shifted prime in this part of the argument,
together with~\eqref{eq: p-1 | q-1}, we conclude the proof.

\section*{Acknowledgements}

The authors are very grateful to Bernd Kellner and Jonathan Sondow for many 
useful comments on the initial version of the manuscript. 

F.L, P.M. and I.S. would like to thank  the Max Planck Institute for Mathematics in Bonn for 
its hospitality and support during the period when this paper was written.

The second  author was   supported in part
 by NRF (South Africa) Grants CPRR160325161141 and an A-rated researcher award
 and by CGA (Czech Republic) Grant 17-02804S. 
 
The fourth  author was  
supported in part by ARC (Australia)  Grant DP140100118.

\end{document}